\documentclass[12pt]{amsart}% COMMENT WHEN MAKING ELS
\usepackage[textwidth=16cm, textheight=22cm]{geometry}%COMMENT WHEN MAKING ELS
\usepackage{hyperref}
\newcommand{\arxiv}[1]{\href{http://www.arXiv.org/abs/#1}{arXiv:#1}}
\usepackage{graphicx}% COMMENT WHEN MAKING ELS
\usepackage{latexsym}

\newtheorem{theorem}{Theorem}[section]
\newtheorem{lemma}[theorem]{Lemma}

\newtheorem{proposition}[theorem]{Proposition}

\newtheorem{corollary}[theorem]{Corollary}
\theoremstyle{remark}

\def\ZZ{\mathbb{Z}}

\def\Rmax{\mathbb{R}_{\max}}

\def\R{\mathbb{R}}

\def\Rp{\R_+}
\def\Rpn{\Rp^n}
\def\Rpnn{\Rp^{n\times n}}

\def\supp{\operatorname{supp}}
\def\spann{\operatorname{span}}
\def\core{\operatorname{core}}

\def\maxmu{\Tilde{\mu}}

\def\goodclass{{\mathcal S}}
\def\crit{{\mathcal C}}
\def\digr{{\mathcal G}}
\def\itoj{i\;-\;j}

\def\mutonu{\mu\;-\;\nu}

\def\coremax{\core}

\def\vmax{V}

\def\spanmax{\spann}

\def\Lmax{\Lambda}

\def\rhomax{\rho}

\def\mrho{M_{\rho}}
\def\amrho{A_{\rho}}
\def\amrhomrho{A_{\mrho\mrho}}

\begin{document}

\title{On the max-algebraic core of a nonnegative matrix}
\thanks{This research was supported by EPSRC grant RRAH15735.
Serge\u{\i} Sergeev also acknowledges the support of RFBR-CNRS grant
11-0193106 and RFBR grant 12-01-00886.}

\author{Peter Butkovi{\v{c}}}
\address{Peter Butkovi{\v{c}}, University of Birmingham,
School of Mathematics, Watson Building, Edgbaston B15 2TT, UK}
\email{P.Butkovic@bham.ac.uk}

\author{Hans Schneider}
\address{Hans Schneider, University of Wisconsin-Madison,
Department of Mathematics, 480 Lincoln Drive, Madison WI 53706-1313,
USA} \email{hans@math.wisc.edu}

\author{Serge\u{\i} Sergeev}
\address{Serge\u{\i} Sergeev, University of Birmingham,
School of Mathematics, Watson Building, Edgbaston B15 2TT, UK}
\email{sergeevs@maths.bham.ac.uk}

\begin{abstract}
The max-algebraic core of a nonnegative matrix is the intersection
of column spans of all max-algebraic matrix powers. Here we
investigate the action of a matrix on its core. Being closely
related to ultimate periodicity of matrix powers, this study leads
us to new modifications and geometric characterizations of robust,
orbit periodic and weakly stable matrices.

\end{abstract}

\maketitle

\section{Introduction}

The concept of the matrix core was introduced by Pullman~\cite{Pul-71} in the
case of nonnegative algebra. Defined as the intersection of nonnegative column spans of matrix powers
\begin{equation}
\label{core}
\core(A)=\bigcap_{t=1}^{\infty} \spann(A^t),
\end{equation}
it was used to derive an alternative geometric proof of the
Perron-Frobenius theorem. Tam and Schneider~\cite{TS-94} developed an
extension of Pullman's results to cone-preserving maps.
The core is also related to the limiting sets of
Markov chains investigated by Sierksma~\cite{Sie-99}, or the limiting sets
of matrix groups described, e.g., by Seneta~\cite{Sen:81} and Hartfiel~\cite{Har:02}.

In a previous work~\cite{BSS-core1} we investigated two cores of a nonnegative
matrix: the ``traditional'' core defined by Pullman, and the ``new'' core
defined in max algebra.   Generalizing an
argument of Pullman we gave a simultaneous proof that in both algebras, $\core(A)$ equals the Minkowski sum of all eigencones
(i.e., cones of nonnegative eigenvectors) of matrix powers.
We also proved the periodicity of the sequence of eigencones of matrix powers and described extremals of the core, in both algebras.
These results can be seen in the framework of unification and parallel development of nonnegative linear algebra and
max algebra. To this end, let us also mention the previous works on generators,
extremals and bases of max cones~\cite{BSS}, and Z-matrix
equations~\cite{BSS-zeq}.

In the present work we are interested in the action of a nonnegative matrix
on its max-algebraic core. We observe that while this action is bijective in the
usual (nonnegative) algebra, it is in general only surjective in the case of max algebra. Further, Butkovi\v{c} et al. previously investigated
certain classes of matrices related to the ultimate periodicity, in
particular, robust matrices~\cite{BCG-09} and weakly stable matrices~\cite{BSS-ws}.
Following the same ideas, Sergeev and Schneider~\cite{SerSch} considered orbit periodic matrices. Here we introduce the core restrictions of these notions. We show that this leads to a
more geometric and in a sense, more transparent characterisation of robust and orbit periodic matrices. In the case of weakly stable matrices, the core
restriction does not yield any alternative characterization, but it gives
a necessary and sufficient condition
for a matrix to be bijective on its core, in max algebra.

The paper is organized as follows. Section 2 is occupied with preliminaries on max-algebraic cyclicity, Frobenius normal form and reducible spectral theory.
Here we also give a review of the results on access relations between classes of
Frobenius normal form of matrix powers, formulated in~\cite{BSS-core1}.
In Section 3 we review the equality between the core and Minkowski sum of
eigencones of matrix powers obtained in~\cite{BSS-core1}
for the case of max algebra, and introduce the classes
of matrices mentioned above. We add here some new observations, on the core of
integer matrices in max-plus algebra, and on the core in max-min algebra.
Section 4 contains our main results on core restrictions of robustness, orbit periodicity and weak stability. Here we develop a geometric characterization of
robust and orbit periodic matrices in terms of the core, and establish an equivalence between bijectivity and the core restriction of weak stability.

\section{Max-algebraic cyclicity and spectral theory}

\subsection{Max algebra: main objects}

By max algebra we understand the set of nonnegative numbers $\Rp$ where the role of
addition is played by taking maximum of two numbers: $a\oplus b:=\max(a,b)$, and the multiplication
is as in the usual arithmetics. This is carried over to matrices and vectors like in the usual linear algebra
so that for two matrices $A=(a_{ij})$ and $B=(b_{ij})$ of appropriate sizes,
$(A\oplus B)_{ij}=a_{ij}\oplus b_{ij}$ and $(A\otimes B)_{ik}=\bigoplus_k a_{ik}b_{kj}$.
%In this paper, max-algebra will always go in parallel with the usual nonnegative linear algebra, to
%compare and unite some elements of both theories.
Notation $A^k$ will stand for the $k$th max-algebraic power.

With a matrix $A=(a_{ij})\in\Rp^{n\times n}$
we associate a weighted (di)graph $\digr(A)$ with the set of nodes
$N=\{1,\dots,n\}$ and set of edges~$E\subseteq N\times N$ containing a pair~$(i,j)$ if and only
if~$a_{ij}\neq 0$; the weight of an edge~$(i,j)\in E$ is defined to be~$w(i,j):=a_{ij}$.
A graph with just one node and no
edge will be called {\em trivial}.

A path $P$ in $\digr(A)$ is a sequence of nodes $i_0,i_1,\ldots,i_t$
such that each pair $(i_0,i_1)$,\\$(i_1,i_2),\ldots,(i_{t-1},i_t)$
is an edge in $\digr(A)$. It has {\em length} $l(P):=t$ and {\em
weight} $w(P):=w(i_0,i_1) \cdot w(i_1,i_2) \cdots w(i_{t-1},i_t)$,
and is called an $\itoj$ path if $i_0=i$ and $i_t=j$. $P$ is called
a {\em cycle} if $i_0=i_t$, and a cycle is called {\em elementary}
if all nodes of the cycle are different.

$A=\left( a_{ij}\right) \in \Rp^{n\times n}$ is
irreducible if $\digr(A)$ is trivial or for any $i,j\in \{1,\ldots, n\}$ there is an
$\itoj$ path. Otherwise $A$ is reducible.

A set $V\subseteq\Rpn$ will be called a {\em max cone} if
1) $\alpha v\in V$ for
all $v\in V$ and $\alpha\in\Rp$, 2) $u\oplus v\in V$ for $u,v\in V$. Max cones are also
known as idempotent semimodules, see~\cite{KM:97,LM-98}.
A max cone $V$ is said to be {\em generated} by $S\subseteq\Rpn$ if each $v\in V$ can be represented
as a max combination $v=\bigoplus_{x\in S} \alpha_x x$ where only finitely many (nonnegative) $\alpha_x$ are different from
zero. When $V$ is generated (we also say ``spanned'') by $S$,
this is denoted $V=\spanmax(S)$. When $V$ is generated by the columns of a matrix $A$,
this is denoted $V=\spanmax(A)$. Max cones are max-algebraic analogues of convex cones.

A vector $z$ in a max cone $V\subseteq\Rpn$ is called an {\em extremal} if $z=u\oplus v$ and $u,v\in V$ imply
$z=u$ or $z=v$. Any finitely generated max cone is generated by its extremals, see Wagneur~\cite{Wag-91} and~\cite{BSS,GK-07} for recent extensions. %(e.g., the tropical Minkowski theorem).

The {\em maximum cycle geometric mean} of~$A$ is defined by
\begin{equation}\label{mcgm}
\rho(A)=\max\{ w(C)^{1/l(C)};\; C \text{ is a cycle in }\digr(A)\} \enspace
\end{equation}
The {\em critical graph} of $A$, denoted by $\crit(A)$, consists of all nodes and edges belonging to the cycles
which attain the maximum in~\eqref{mcgm}. The set of such nodes will be called {\em critical} and denoted $N_c$;
the set of such edges will be called {\em critical} and denoted $E_c$. Observe that the critical graph, defined as above,
consists of several strongly connected subgraphs of $\digr(A)$. Maximal such subgraphs are the
{\em strongly connected components} of $\crit(A)$. %also callled just {\em critical components} of $A$.

If for $A\in\Rpnn$ we have
$A\otimes x=\rho x$
with $\rho\in\Rp$ and a nonzero $x\in\Rpn$, then $\rho$ is a {\em max(-algebraic) eigenvalue} and $x$ is a {\em max(-algebraic)
eigenvector} associated with $\rho$. The set of
max eigenvectors $x$ associated with~$\rho$, with the zero vector adjoined to it, is a max cone further denoted
by $\vmax(A,\rho)$. It is called the {\em eigencone} of $A$ associated with
$\rho$.

In general, a matrix $A\in\Rpnn$ may have several max eigenvalues. The greatest max eigenvalue
is equal to $\rho(A)$ (see~\cite{BCOQ,But:10,CG:79,HOW:05}), and it is called the
{\em principal eigenvalue}. The corresponding eigencone is called the
{\em principal eigencone}. It is also known that if $A$ is irreducible then
$\rho(A)$ is the only eigenvalue, which we call the {\em max-(algebraic) Perron root} of $A$.

\if{
An irreducible $A\in\Rpnn$ has a unique max-algebraic eigenvalue equal to the
m.c.g.m.~$\lambda(A)$~\cite{BCOQ,But:10,CG:79,HOW:05}.
In general $A$ may have several max eigenvalues, which we will describe later. The greatest such eigenvalue will be
denoted by $\rhomax(A)$, and called the {\em principal eigenvalue} of $A$, and
the {\em principal eigencone} is the eigencone of $A$ associated with it.
 In the irreducible case, the unique max eigenvalue $\rhomax(A)=\lambda(A)$
will also be called the {\em max(-algebraic) Perron root}.
}\fi

There is an explicit
description of $\vmax(A,\rhomax(A))$, see Theorem~\ref{t:FVmaxalg} below.
It uses the {\em Kleene star}
\begin{equation}
\label{def:kleenestar}
A^*=I\oplus A\oplus A^2\oplus A^3\oplus\ldots,
\end{equation}
where $I$ denotes the identity matrix. Series~\eqref{def:kleenestar} converges if and only if
$\rhomax(A)\leq 1$, in which case $A^*=I\oplus A\oplus\ldots\oplus A^{n-1}$.
Note that if $\rhomax(A)\neq 0$, then $\rhomax(A/\rhomax(A))=1$, and so
$(A/\rhomax(A))^*$ always converges.

The {\em path interpretation} of max-algebraic matrix powers $A^{\otimes l}$
is that each entry $a^{\otimes l}_{ij}$ is equal to the greatest weight of
$\itoj$ paths with length $l$. Consequently, for $i\neq j$, the entry $a^*_{ij}$ of $A^*$ is equal to the
greatest weight of $\itoj$ paths (with no length restrictions).

For a strongly connected graph $\digr$ define its {\em cyclicity} $\sigma$ as the gcd of the
lengths of all elementary cycles and the cyclicity of a trivial graph to be $1$.
For a (general) graph containing several
maximal strongly connected components (such as the critical graph $\crit(A)$), cyclicity is defined as the lcm of the cyclicities
of the strongly connected components. A graph with cyclicity $1$ is called {\em primitive}.

The following result demonstrates importance
of cyclicity of critical graph in max algebra.

\begin{theorem}[Cyclicity Theorem, Cohen~et~al.~\cite{CDQV-83}]
\label{t:Cycl1}
Let $A\in\Rpnn$ be irreducible and let $\sigma$ be the cyclicity of $\crit(A)$. Then $\sigma$ is the smallest $p$ such that there exists
$T(A)$ with $A^{t+p}=\rhomax^{p}(A) A^{t}$ for all $t\geq T(A)$.
\end{theorem}

This result is closely related to the theory of graph exponents as presented,
for instance, in Brualdi-Ryser~\cite{BR}. In particular, it exploits the
following number-theoretic lemma due to Schur and Frobenius:

\begin{lemma}[\cite{BR}, Lemma 3.4.2]
\label{l:schur} Let $n_1,\ldots, n_m$ be integers and let\\
$k=$gcd$(n_1,\ldots,n_m)=k$. Then there exists a number $T$ such that
for all integers $l$ with $kl\geq T$, we have $kl=t_1n_1+\ldots
+t_mn_m$ for some $t_1,\ldots,t_m\geq 0$.
\end{lemma}

\subsection{Ultimate periodicity and immediate periodicity}

Theorem~\ref{t:Cycl1} also shows that the periodicity of %matrix/vector/scalar
sequences is crucial in max algebra. We will need the
following formal definitions.

A sequence of scalars/vectors/matrices $\{X_t\}_{t\geq 1}$ over $\Rp$ is called
{\em ultimately periodic with the period $\tau$ growth rate $r\in\Rp$} if
$\tau$ is the smallest integer $p\geq 1$ such that there exists $T$ with
$X_{t+p}=r^p X_t$ for all $t\geq T$.

A sequence of max cones $\{V_t\}_{t\geq 1}$ in $\Rpn$ is called
{\em ultimately periodic with the period $\sigma$} if
$\tau$ is the smallest integer $p\geq 1$ such that there exists $T$ with
$V_{t+p}=V_t$ for all $t\geq T$.

An ultimately periodic sequence of scalars/vectors/matrices $\{X_t\}_{t\geq 1}$ over $\Rp$
or max cones $\{V_t\}_{t\geq 1}$ in $\Rpn$ is called {\em immediately periodic} or just
{\em periodic}, if $T$ can be set to $1$ in the definitions above.

In terms of the ultimate periodicity, Theorem~\ref{t:Cycl1} can be
formulated as follows: for any irreducible nonnegative matrix $A\in\Rpnn$, the sequence of matrix powers $\{A^t\}_{t\geq 1}$ is ultimately periodic with the period equal to the cyclicity of critical graph, and with the growth rate equal to the m.c.g.m. of $A$.

Strictly speaking, the ultimate periodicity is also characterized by the smallest
possible parameter $T$ in the above definitions, which is called the {\bf defect} of sequence or its {\em periodicity threshold}.
The reader is referred to, e.g., Akian, Gaubert, Walsh~\cite{AGW-05} and Hartman, Arguelles~\cite{HA-99} for various bounds and algorithms estimating the value of $T(A)$ in Theorem~\ref{t:Cycl1}.

We also note that for a general reducible matrix $A\in\Rpnn$, not all the sequences $\{a_{ij}^{(t)}\}_{t\geq 1}$ for $i,j\in\{1,\ldots,n\}$, are ultimately
periodic in the sense of the definition given above. Such sequences can be decomposed into ultimately periodic subsequences with different growth rates, and the reader is referred to De Schutter~\cite{BdS}, Gavalec~\cite{Gav:04} and
Moln\'{a}rov\'{a}~\cite{Mol-05} for more details. However, we do not need such general
results and details in the present paper.

\subsection{Other idempotent algebras}

Max-plus algebra $\R_{\max}$ is defined over $\R\cup\{-\infty\}$, the set of real numbers
completed with the ``least'' element $-\infty$. It is equipped with operations
$a\oplus b:=\max(a,b)$ and $a\otimes b:=a+b$, where it is assumed that
$a\otimes -\infty=-\infty\otimes a= -\infty$ and $a\oplus -\infty=
-\infty\oplus a= a$. When applying the exponential mapping $x\mapsto e^x$, the max-plus algebra is seen to be isomorphic to the max(-times) algebra.
% considered
%in~\cite{BSS-core1} and the present paper.

The max-plus algebra
and linear-algebraic problems over it remind us of the practical
applications of idempotent/tropical mathematics in scheduling and discrete
event systems~\cite{BCOQ, CG:79}, optimization and mathematical physics~\cite{KM:97}. In particular, for the Perron root of an irreducible matrix
$A\in\R_{\max}^{n\times n}$, we see applying the logarithmic transform
to~\eqref{mcgm} that
\begin{equation}
\label{mcm}
\rho(A)=\max_{k=1}^n \max_{i_1,\ldots,i_k}
\frac{a_{i_1i_2}+\ldots+a_{i_ki_1}}{k}.
\end{equation}
This is the maximum cycle mean of $A$, i.e., the solution to a linear programming problem (\cite{CG:79}, see also~\cite{But:10}).

In the present paper, we still choose to work with the max-times version
for its theoretical convenience.
% and in accord with~\cite{BSS-core1}
%(though this is not principal here).
The development of max-plus linear algebra is very similar to that of
max(-times) linear algebra.

However, we will touch upon $\ZZ_{\max}$: extended set of integers
$\ZZ\cup\{-\infty\}$ equipped with the same max-plus arithmetics. This is
a subsemiring of $\R_{\max}$, whose ``dual'' $\ZZ_{\min}$ was
treated in a work of Simon~\cite{Sim-94}, with
motivations in computer science and automata theory.

We will also touch upon the {\em max-min algebra}: the
interval $[0,1]$ equipped with ``addition'' $a\oplus b:=\max(a,b)$ and
``multiplication'' $a\otimes b=\min(a,b)$. For more information on the
linear algebra over this semiring, also known as {\em fuzzy linear algebra}, the
reader is referred to, e.g., Gavalec~\cite{Gav:04}. The only observation that we will
later use, is that unlike in max algebra, the ``multiplication'' here does not
create any new numbers. In particular, it implies that the sequence
$\{A^t\}_{t\geq 1}$ is ultimately periodic for each matrix $A$ with the
growth rate $1$.

\subsection{Frobenius normal form}

Every matrix $A=(a_{ij})\in \Rp^{n\times n}$ can be transformed by
simultaneous permutations of the rows and columns in almost linear
time to
a {\em Frobenius Normal Form} (FNF)~\cite{BP,BR} %\cite{Handbook}
\begin{equation}
\left(
\begin{array}{cccc}
A_{11} & 0 & ... & 0 \\
A_{21} & A_{22} & ... & 0 \\
... & ... & A_{\mu\mu} & ... \\
A_{r1} & A_{r2} & ... & A_{rr}%
\end{array}%
\right) ,  \label{fnf}
\end{equation}%
where $A_{11},...,A_{rr}$ are irreducible square submatrices of $A$,
corresponding to the partition $N_1\cup\ldots\cup N_r=N$. Generally,
$A_{KL}$ denotes the submatrix of $A$ extracted from rows with indices in $K\subseteq N$
and columns with indices in $L\subseteq N$, and
$A_{\mu\nu}$ is a shortcut for $A_{N_{\mu}N_{\nu}}$.

Consider the graph associated with $A_{\mu\mu}$. It is the same as the graph
induced by $N_{\mu}$:
$
\digr(A_{\mu\mu})=(N_{\mu},E\cap (N_{\mu}\times N_{\mu})).
$
It follows that each of
the graphs $\digr(A_{\mu\mu})$ $(\mu=1,...,r)$ is strongly
connected and an arc from $N_{\mu}$ to $N_{\nu}$ in $\digr(A)$ may exist only
if $\mu\geq \nu.$

If $A$ is in the Frobenius Normal Form \eqref{fnf} then the {\em
reduced graph}, denoted $R(A)$, is the (di)graph whose nodes correspond to
$N_{\mu}$, for $\mu=1,\ldots,r$, and the set of arcs is
$
\{(\mu,\nu);(\exists k\in N_{\mu})(\exists \ell \in N_{\nu})a_{k\ell }>0\}).
$

The
nodes of $R(A)$ are {\em marked} by the corresponding max-algebraic eigenvalues
(Perron roots) denoted by $\rhomax_{\mu}:=\rhomax(A_{\mu\mu})$.
These nodes will be called {\em classes} of $A$.
We naturally attribute to a class $\mu$ also
the graph $\digr(A_{\mu\mu})$ with set of nodes $N_{\mu}$ and cyclicity $\sigma_{\mu}$.
%We will also sometimes refer to $N_{\mu}$ as a class of $A$, by abuse of terminology.

We say that a class $\mu$ is {\em trivial} if $A_{\mu\mu}$ consists of a single diagonal
zero entry, i.e., when $\digr(A_{\mu\mu})$ is trivial.
Class $\mu$ {\em accesses} class
$\nu$, denoted $\mu\to\nu$, if there exists a $\mutonu$ path in $R(A)$.
A class is called {\em initial}, resp. {\em final}, if it is not accessed by, resp. if it does not
access, any other class. Node $i$ accesses class $\nu$, denoted by $i\to\nu$, if $i$ belongs to a
class $\mu$ such that $\mu\to\nu$.

Simultaneous permutations of the rows and columns of $A$ are
equivalent to calculating $P^{-1}AP,$ where $P$ is a
permutation matrix. Such transformations do not change the
eigenvalues, and the eigenvectors before and after such a
transformation only differ by the order of their components. Hence we will assume without loss of generality
that $A$ is in FNF~(\ref{fnf}).

\subsection{Reducible spectral theory}

A class $\nu$ of $A$ is called a {\em spectral class} of $A$ associated with eigenvalue $\rho\neq 0$, or sometimes
$(A,\rho)$-spectral class for short, if
\begin{equation}
\label{e:speclass}
\rhomax_{\nu}=\rho,\ \text{and}\ \mu\to\nu\ \text{implies}\ \rhomax_{\mu}\leq\rho_{\nu}.
\end{equation}
Note that there may be several spectral classes associated with the same eigenvalue.

Denote by  $\Lmax(A)$, the set of {\bf nonzero} eigenvalues of $A\in\Rpnn$.
 The following description is standard.
%In nonnegative algebra, see Schneider~\cite{Sch-86}, Th.~3.7~(i)-(b),
%where spectral classes are called {\bf distinguished}. In max algebra, see, e.g.,
%Gaubert~\cite{Gau:92}, Ch.~IV, Coro. 2.2.5, or Butkovi\v{c}~\cite{But:10}~Th.~4.5.4.

\begin{theorem}[\cite{But:10}~Th.~4.5.4, \cite{Gau:92}]
\label{t:spectrum}
Let $A\in\Rpnn$. Then\\
$\Lmax(A)=\{\rhomax_{\nu}\neq 0;\  \nu\ \text{is spectral}\}=
\{\rhomax_{\nu}\neq 0;\ \forall\mu,\ \mu\to\nu\ \text{implies}\ \rhomax_{\mu}\leq\rhomax_{\nu}\}$.
\end{theorem}

For each $\rho\in\Lambda(A)$ define %(in max algebra and nonnegative algebra simultaneously), define
\begin{equation}
\label{amrho}
\begin{split}
&\amrho:=\rho^{-1}
\begin{pmatrix}
0 & 0\\
0 & \amrhomrho
\end{pmatrix},\ \text{where}\\
& \mrho:=\{i;\ i\to\nu,\; \nu\ \text{is $(A,\rho)$-spectral}\}\,\enspace .
\end{split}
\end{equation}

The next proposition
allows us to reduce a general eigencone to the
case of a principal eigencone. We will assume an appropriate
ordering of indices (i.e., nodes of the corresponding graph). Such assumptions will sometimes be
made also in the sequel, without special mention of them.

%In nonnegative algebra, it follows from
%Schneider~\cite{Sch-86}, Th.~3.7~part (ii), and
%for the max-algebraic part see~\cite{Gau:92}~
%Ch.~IV~Sect.2.3, and also~\cite{But:10,BCG-09}.
%We will use that the set $\mrho$
%defined in~\eqref{amrho} is {\em initial} in $\digr(A)$, i.e., all nodes having access to a node of
%$\mrho$ are in $\mrho$.

\begin{proposition}[\cite{But:10,Gau:92}]
\label{p:vamrho}
For $A\in\Rpnn$ and each $\rho\in\Lambda(A)$, we have
$V(A,\rho)=V(\amrho,1)$, where $1$ is the principal eigenvalue
of $\amrho$.
\end{proposition}

Using Proposition~\ref{p:vamrho}, we
define the {\em critical graph associated with} $\rho\in\Lmax(A)$ as the critical graph of
$\amrho$. The strongly connected components of critical graphs
associated with $\rho$, for all $\rho\in\Lmax(A)$, will be called the {\em critical components of $A$}.
%Gavalec~\cite{Gav:00} calls them highly connected components.
These components will be (similarly as the
classes of FNF) denoted by $\Tilde{\mu}$, with the node set $N_{\Tilde{\mu}}$.

We now describe the principal eigencones in max algebra. By means of Proposition~\ref{p:vamrho}, this description can be obviously extended
to the case of general eigencones. Here, the vectors $x_{\Tilde{\mu}}$ are
{\bf full-size}. Recall that support of $x\in\Rpn$, denoted by $\supp x$, is the set of indices $i$ with $x_i\neq 0$.

\begin{theorem}[\cite{But:10}~Th.~4.3.5, \cite{BSS-core1}~Th.~2.11, \cite{Gau:92} ]
\label{t:FVmaxalg}
Let $A\in\Rpnn$ have $\rhomax(A)=1$.
\begin{itemize}
\item[(i)] Each component $\maxmu$ of $\crit(A)$ corresponds to
an eigenvector $x_{\maxmu}$ defined as one of the columns $A^*_{\cdot i}$
with $i\in N_{\maxmu}$, all columns with $i\in N_{\maxmu}$ being multiples of each other.
\item[(i')] Each component $\maxmu$ of $\crit(A)$ is contained in a (spectral) class $\mu$
with $\rhomax_{\mu}=1$, and the support of each $x_{\maxmu}$ of (i) consists of
all indices in all classes that have access to $\mu$.

\item[(ii)] $\vmax(A,1)$ is generated by $x_{\maxmu}$ of (i), for $\maxmu$ ranging over all
components of $\crit(A)$.
\item[(iii)] $x_{\maxmu}$ of (i) are extremals in $\vmax(A,1)$. (Moreover, $x_{\maxmu}$ are
strongly linearly independent in the sense of~\cite{But-03}.)
\end{itemize}
\end{theorem}

\subsection{Access relations in matrix powers}

In~\cite{BSS-core1} it was demonstrated that access relations and
spectral classes of all matrix powers are essentially the same, and
that the case of an arbitrary eigenvalue reduces to the case of the
principal eigenvalue. We start from the following basic result.

\begin{lemma}[\cite{BSS-core1}, Lemma 5.1]
\label{l:sameperron}
Let $A$ be irreducible with the (unique) eigenvalue $\rho$,
let $\digr(A)$ have cyclicity $\sigma$ and $t$ be a positive integer.
Then, $A^t$ is a direct sum of gcd$(t,\sigma)$ irreducible blocks
with eigenvalues $\rho^t$, and $A^t$ does not have eigenvalues other than $\rho^t$.
The cyclicity of
each block is $\sigma/$gcd$(t,\sigma)$. In particular, all
blocks of $A^{\sigma}$ are primitive.
 %The set of all indices in these blocks is $\{1,\ldots,n\}$,
\end{lemma}

Recall that each class $\mu$ of $A$ %with cyclicity $\sigma$
corresponds to an
irreducible submatrix $A_{\mu\mu}$. It is easy to see that
$(A^t)_{\mu\mu}=(A_{\mu\mu})^t$ for any positive integer $t$. Suppose that the
cyclicity of $\digr(A_{\mu\mu})$ is $\sigma$. Applying Lemma~\ref{l:sameperron}
to $A_{\mu\mu}$ we see that $\mu$ gives rise to gcd$(t,\sigma)$ classes
in $A^t$, which are said to be {\em derived} from their common {\em ancestor} $\mu$.
The classes of $A^t$ and $A^l$ derived from the common ancestor will be called {\em related}.
Note that this is an equivalence relation
on the set of classes of all powers of $A$.

It can be checked that the same notions can be defined for the
components of critical graphs, see~\cite{BSS-core1}.

Let us recall the following results on the similarity of access relations in matrix powers.

\begin{lemma}[\cite{BSS-core1}, Lemma 5.3]
\label{l:sameaccess}
For all $t,l\geq 1$ and $\rho>0$, an index $i\in\{1,\ldots,n\}$ accesses (resp.  is accessed by) a class with Perron root $\rho^t$ in $A^t$ if and only if
it accesses (resp. is accessed by) a related class with Perron root $\rho^l$ in $A^l$.
\end{lemma}

A similar result holds for the strongly connected components of
the critical graphs of matrix powers, see~\cite{BSS-core1}, Theorem 3.3.
We will only need  the following observation,
which we formulate for the {\em critical matrix} $A_{C}=(a_C)_{ij}$ defined by
\begin{equation}
(a_C)_{ij}=
\begin{cases}
1, & \text{if $(i,j)\in E_c$},\\
0, & \text{otherwise}.
\end{cases}
\end{equation}
\begin{lemma}
\label{l:critpowers}
Let $A\in\Rpnn$. Then $(A_C)^t=(A^t)_C$.
\end{lemma}
This observation will allow us to apply Lemma~\ref{l:sameperron} to
the critical graphs $\crit(A^t)$ of matrix powers, that is, to the
critical matrices $(A^t)_C$.

All eigenvalues and spectral classes of matrix powers
are derived from those of $A$.

\begin{theorem}[\cite{BSS-core1}, Th.~5.4, Coro.~5.5]
\label{t:samespectrum}
Let $A\in\Rpnn$ and $t\geq 1$.
\begin{itemize}
\item[(i)]$\Lambda(A^t)=\{\rho^t;\  \rho\in\Lambda(A)\}.$
\item[(ii)] For each spectral class $\mu$ of $A$
with cyclicity $\sigma$ there are gcd$(t,\sigma)$ spectral classes
of $A^t$ derived from it. Conversely, each spectral class of
$A^t$ is derived from a spectral class of $A$.
\end{itemize}
\end{theorem}

As in the case of eigencones of a matrix, when
working with $V(A^t,\rho^t)$ we can assume that $\rho=1$ is the principal eigenvalue
of $A$, and hence of all $A^t$.

\begin{theorem}[\cite{BSS-core1}, Th.~5.7]
\label{t:reduction}
Let $A\in\Rpnn$, $t\geq 1$ and $\rho\in \Lambda(A)$.
\begin{itemize}
%\item[(i)] The set of all indices having access to the spectral classes of $A^k$
%with the eigenvalue $\rho^k$ equals $M_{\rho}$, for each $k$.
\item[(i)] $(A^t)_{\mrho\mrho}=\left(\rho^t (\amrho)^t\right)_{\mrho\mrho}$.
\item[(ii)] $V(A^t,\rho^t)=V((\amrho)^t,1)$.
%\item[(iii)] The eigenvalue $\rho^k$ is principal in  $(A_{M_{\rho}M_{\rho}})^k=(A^k)_{M_{\rho}M_{\rho}}$.
\end{itemize}
\end{theorem}

\section{Core and eigenvectors}
\label{s:core}

In this section we recall the main results of~\cite{BSS-core1}
on the matrix core, and its finite stabilization in special cases.
We also consider the case of integer entries and, briefly, the core in
max-min algebra and other algebras where
every sequence of matrix powers $\{A^t\}_{t\geq 1}$ is ultimately periodic.
%the ultimate periodicity takes
%place for any matrix.

\subsection{Core in general}

The purpose of~\cite{BSS-core1}, Section~4 was to show, by adapting
an argument of Pullman~\cite{Pul-71} to max algebra, that the max-algebraic core can be also represented
as the (Minkowski) sum of the eigencones of matrix powers, that is,
\begin{equation}
\label{e:core-mink}
\core(A)=\bigoplus_{t\geq 1,\rho\in\Lambda(A)} V(A^t,\rho^t).
\end{equation}
By definition, on the r.h.s. of~\eqref{e:core-mink} we have a
max cone consisting of all combinations $\bigoplus_t \alpha_t y^{(t)}$ with
$y^{(t)}\in V(A^t,\rho^t)$ and finite number of nonzero $\alpha_t$.

It was also shown in \cite{BSS-core1}, Section~7 that the sequence
of eigencones of matrix powers is periodic.
To describe this periodicity, the following notation was introduced:\\
1. $\sigma_{\rho}$: cyclicity of the critical graph associated with
an eigenvalue $\rho\in\Lambda(A)$,\\
2. $\sigma_{\Lambda}$: the lcm of all $\sigma_{\rho}$, over $\rho\in\Lambda(A)$,\\
and the Minkowski sums of all eigencones of matrix powers were
considered:
\begin{equation}
\label{e:vat}
V^{\Sigma}(A^t):=\bigoplus_{\rho\in\Lambda(A)} V(A^t,\rho^t).
\end{equation}
The periodicity can be described as follows.
\begin{theorem}[\cite{BSS-core1}, Main Theorem 2]
\label{t:girls}
Let $A\in\Rpnn$ and $\rho\in\Lmax(A)$. Then
the sequence of max cones $\{\vmax(A^t,\rho^t)\}_{t\geq 1}$
is immediately periodic with the period $\sigma_{\rho}$,
and $\vmax(A^t,\rho^t)\subseteq \vmax(A^{\sigma_{\rho}},\rho^{\sigma_{\rho}})$ for all $t$.
\end{theorem}

\begin{theorem}[\cite{BSS-core1}, Main Theorem 2]
\label{t:girls-general} Let $A\in\Rpnn$. Then the sequence of max
cones $\{\vmax^{\Sigma}(A^t)\}_{t\geq 1}$ is immediately periodic
with the period $\sigma_{\Lambda}$, and
$\vmax^{\Sigma}(A^t)\subseteq \vmax^{\Sigma}(A^{\sigma_{\Lambda}})$
for all $t$.
\end{theorem}

In particular,
$V(A^t,\rho^t)\subseteq V(A^{\sigma_{\Lambda}},\rho^{\sigma_{\Lambda}})$ for all
$\rho\in\Lambda(A)$, and hence
\begin{equation}
\label{e:vasigma}
\bigoplus_{t\geq 1,\rho\in\Lambda(A)} V(A^t,\rho^t)= \bigoplus_{\rho\in\Lambda(A)}
V(A^{\sigma_{\Lambda}},\rho^{\sigma_{\Lambda}})=V^{\Sigma}(A^{\sigma_{\Lambda}}).
\end{equation}

The central result of~\cite{BSS-core1} can be formulated
as follows:

\begin{theorem}[\cite{BSS-core1}, Main Theorem 1]
\label{t:core}
Let $A\in\Rpnn$. Then\\
$\core(A)=\bigoplus_{\rho\in\Lambda(A)}
V(A^{\sigma_{\rho}},\rho^{\sigma_{\rho}})
=V^{\Sigma}(A^{\sigma_{\Lambda}})$.
\end{theorem}

Note that the inclusion  $\bigoplus_{\rho\in\Lambda(A)} V(A^{\sigma_{\rho}},\rho^{\sigma_{\rho}})\subseteq\core(A)$ holds since each vector in $V(A^{\sigma_{\rho}},\rho^{\sigma_{\rho}})$ is in
$\spann(A^{t\sigma_{\rho}})$ for all $t\geq 1$, hence in $\core(A)$.

The proof of
the opposite inclusion relied on the facts collected in
Lemma~\ref{l:actioncore} below. Here, a vector $v\in\Rpn$ is called {\em scaled} if
$||v||:=\max\limits_{i=1}^n v_i=1$, and $A$ induces a mapping on the scaled
vectors of $\core(A)$ by $v\mapsto Av/||Av||$.

\begin{lemma}
\label{l:actioncore}
Let $A\in\Rpnn$, then
\begin{itemize}
\item[(i)] $\core(A)$ is generated by no more than $n$ vectors,
\item[(ii)] the mapping induced by $A$ on $\core(A)$ is a surjection,
\item[(iii)] the mapping induced by $A$ on the scaled extremals of $\core(A)$
is a permutation (i.e., a bijection).
\end{itemize}
\end{lemma}

Note that in {\bf nonnegative linear algebra} $A$ is always bijective on its core. This follows since
the extreme generators of the nonnegative core are linearly independent in the usual linear algebra, so
$A$ just permutes and rescales the elements of basis of the linear-algebraic span of the nonnegative core.

In {\bf max algebra}, $A$ is in general non-bijective on $\core(A)$, and we later describe when $A$ is bijective. Moreover, {\em finite stabilization of the core}
(when $\core(A)=\spann(A^t)$ for all
large enough $t$) is often observed. Note that in this case, Theorem~\ref{t:core} follows
immediately.

\if{
If $z$ is an eigenvector of $A^{\sigma}$, then for each class of the
Frobenius normal form of $A^{\sigma}$, either all indices of the
class are in $\supp(z)$, or none. This extends to all combinations of such
eigenvectors, i.e., to all vectors of $\core(A)$, and we also have the following
observation.
}\fi

\subsection{Ultimate periodicity and finite stabilization}
\label{ss:maxalg-easy}

In max algebra there are wide classes of matrices $A\in\Rpnn$ where we have finite stabilization of the core. We list some of them.
\begin{itemize}
\item $\goodclass_1:$ {\bf Irreducible matrices}.
\item $\goodclass_2:$ {\bf Ultimately periodic matrices.} This is when we have
$A^{t+\sigma}=\rho^{\sigma} A^t$ for all sufficiently large $t$,
with $\rho=\rhomax(A)$. As shown in~\cite{MP-00}, this happens if and only if the Perron roots of all nontrivial classes
of $A$ equal $\rho(A)$.
\item $\goodclass_3:$ {\bf Robust matrices.} For any nonzero vector $x\in\Rpn$ the orbit $\{A^tx\}_{t\geq 1}$ hits an eigenvector of $A$,
implying that the whole remaining part of the orbit consists of multiples of that eigenvector. The notion
of robustness was introduced and studied in~\cite{BCG-09}, and it will be revisited below.
\item $\goodclass_4:$ {\bf Orbit periodic matrices:} For any nonzero vector $x\in\Rpn$ the orbit $\{A^tx\}_{t\geq 1}$ hits an eigenvector of $A^{\sigma}$,
implying that the remaining part of the orbit is periodic (with some growth rate).
See~\cite{SerSch}, Section 7 and below for
characterization.
\item $\goodclass_5:$ {\bf Column periodic matrices.} This is when for any $i=1,\ldots,n$ we have
$(A^{t+\sigma})_{\cdot i}=\rho_i^{\sigma} A^t_{\cdot i}$ for all large enough $t$
and some
$\rho_i$.
\end{itemize}

Observe that $\goodclass_1\subseteq\goodclass_2\subseteq\goodclass_4\subseteq\goodclass_5$ and
$\goodclass_3\subseteq\goodclass_4$ (see, e.g., \cite{BSS-core1}, Section 4).
To see that $\spanmax(A^t)=\coremax(A)$ for all large enough $t$ in all these
cases, observe that in the column periodic case all sequences of
columns end up with periodically repeating eigenvectors of $A^{\sigma}$, which implies that $\spanmax(A^t)\subseteq\coremax(A)$ for all large enough $t$, and hence also $\spanmax(A^t)=\coremax(A)$.
Thus, finite stabilization of the core occurs in all these classes.

\subsection{Integer max-plus case}

Let us consider the semiring $\ZZ_{\max}$ defined on the set $\ZZ\cup\{-\infty\}$ with
$\oplus=\max$ and $\otimes=+$. This is a subsemiring of the max-plus algebra $\Rmax$, which
is isomorphic to the max(-times) algebra that we consider. Essentially this is the case
of integer matrices in $\Rmax^{n\times n}$. For such integer matrices Theorem~\ref{t:core} holds by isomorphism,
with max-plus arithmetics. We will show that the generators of the core of such matrices are in $\ZZ_{\max}^n$, i.e., that they have integer components.
This implies that Theorem~\ref{t:core} is also true in $\ZZ_{\max}$.
\begin{theorem}
\label{t:integer}
Let $A\in\Rmax^{n\times n}$ have integer entries only. Then $\coremax(A)$
is generated by integer vectors.
\end{theorem}
\begin{proof}
We show that $\vmax(A^{\sigma},\rho^{\sigma})$, for any $\rho\in\Lambda(A)$ and $\sigma=\sigma_{\rho}$, is generated by integer vectors.
For this, we need to show that all $\rho\in\Lambda(A^{\sigma})$,
that is, the maximal cycle means in the spectral blocks of $A$, are
integer.
More precisely, each $\rho\in\Lambda(A^{\sigma})$ can be expressed as
\begin{equation}
\label{e:mcm}
\rho=\frac{a_{i_1i_2}^{(\sigma)}+\ldots+a^{(\sigma)}_{i_mi_1}}{m},
\end{equation}
where $(i_1,\ldots,i_m)$ is any elementary cycle in a strongly connected component of the critical
graph corresponding to $\rho$.  Eqn.~\eqref{e:mcm} implies that
$\rho=K/M$, where $M$ is the lcm of all lengths of elementary cycles
$(i_1,\ldots,i_m)$, and $K$ is a multiple of $M/m$ for any such cycle.
All strongly connected components of the critical graph of $A^{\sigma}$ are primitive, hence the gcd of all denominators in~\eqref{e:mcm}, taken over all elementary cycles $(i_1,\ldots,i_m)$ in a critical component, is $1$.

Next we use the following observation from the elementary number theory:
if the numbers $m_i$, for $i=1,\ldots,\ell$ are coprime (i.e., gcd$(m_i)=1$),
and $M=$ lcm$(m_i)$, then lcm$(M/m_i)=M$. To prove this observation, we
set $M'=$ lcm$(M/m_i)$. Clearly, $M'$ divides $M$. If $M'\neq M$ then
$M=\alpha M'$ with $\alpha>1$. But then for each $i$ there is an integer
$k_i\geq 1$ such that
$$ m_i=\alpha\frac{M'}{M/m_i}=k_i\alpha,\;\alpha>1,$$
hence $\alpha$ divides gcd$(m_i)$ contradicting gcd$(m_i)=1$.

Since $K$ is a multiple of
$M/m_i$ and lcm$(M/m_i)=M$, we obtain that $K$ is a multiple of $M$ so $\rho$ in~\eqref{e:mcm} is integer.
Since all $\rho\in\Lambda(A^{\sigma})$ are integer, and since extremals of the cone
are columns of $(A^{\sigma}_{\rho})^*$ for $\rho\in\Lambda(A)$ (by Theorem~\ref{t:core} and
Theorem~\ref{t:FVmaxalg}), the result follows.
\end{proof}

This result can be also deduced from the Cyclicity
Theorem~\ref{t:Cycl1}. By that result, $\rho^{\sigma}$ satisfies
$A^{t+\sigma}=\rho^{\sigma} A^t$ for all large enough $t$, for $A$
irreducible. If $A$ has only integer entries then so do all the
powers of $A$, and we deduce that $\rho^{\sigma}$ is integer. In the
case of reducible $A$, this argument is applied to the powers of
each submatrix of $A$ that corrsesponds to a spectral class (observe
that the cyclicity of the critical graph in each spectral class
divides $\sigma$, which is the lcm of all such cyclicities).

\subsection{Core in max-min algebra}

Max-min algebra gives an example where
the ultimate periodicity of $A^t$ takes place for all $A$, with the growth rate $1$~\cite{Gav:04}.
Indeed, the operations in this semiring are such that all entries of $A^t$ are among the entries of
$A$, so they start to repeat after some time, with a period bounded by $n^2$.
See~\cite{Gav:04} for more information. Hence in this case we have an
analogue of Theorem~\ref{t:core} where $\bigoplus_{\rho\in\Lambda(A)} V(A^t,\rho^t)$ must be replaced with
$V(A^t,1)$. So
\begin{equation}
\label{e:coremaxmin}
\core(A)=\bigoplus_{k\geq 1} V(A^k,1)
\end{equation}
in max-min algebra and in any other algebra where the sequence $\{A^t\}$ is ultimately periodic with
growth rate $1$ for all $A$.  Note that in max-min algebra each number is eigenvalue, however
the corresponding eigenvectors do not belong to the core unless they can also be associated with
the eigenvalue $1$.

\section{Action of a matrix on the core}
\label{s:coremax}

This section contains the main results of the present paper.
Here we study the mapping induced by a matrix on its core in
max algebra.

\subsection{Finite stabilization}
\label{ss:finite-stab}

We first extend the Cyclicity Theorem (Theorem~\ref{t:Cycl1})
to the case of a {\bf spectral index}, i.e., an index that belongs to a spectral class with $\rho_{\mu}=\rho(A)$. Note that in the case when $A$ is irreducible there is only one class, which is spectral, and therefore every index is spectral in this case.

%The statement below can be also seen as a slight extension of the Cyclicity Theorem, which
%can be shown following the same lines.

%TODO: Locate Theorem~\ref{t:cyclicity} in the literature?

\begin{theorem}
\label{t:cyclicity}
Let $A\in\Rpnn$, and let $j$ be an index in a spectral class.
Then the sequence of columns $\{A^t_{\cdot j}\}_{t\geq 1}$ is ultimately periodic.
\end{theorem}
\begin{proof}
We can assume without loss of generality
that the greatest eigenvalue (i.e., the m.c.g.m.) of $A$ is
$1$. For now we will also assume that the critical graph $\crit(A)$ is primitive, and that the spectral class containing $j$ is associated with the eigenvalue $1$, i.e., with the
greatest eigenvalue. We will show how to omit these two assumptions in the end of the proof.

Let $i$ be an index with access to $j$. We show that
$a_{ij}^{(t)}$ is constant at all large enough $t$. Note that when
$i$ does not have access to $j$, we have $a_{ij}^{(t)}=0$ for all $t$.

Denote by $\Pi_1$  the set of paths that connect $i$ to $j$ via a critical index, with length
bounded by $2(n-1)$, and denote by $w(\Pi_1)$ the biggest weight of paths in $\Pi_1$.
Observe that since $j$ is in a spectral class,
there are paths connecting $i$ to $j$ via a critical index.
Indeed, there are critical nodes in the strongly connected component of $\digr(A)$ containing $j$. Any such critical node can be connected to $j$ by a path, and back.
The resulting cycle can be appended to the access path from $i$ to $j$, forming
a path that connects $i$ to $j$ via a critical index.

Denote by $\Pi_2$ the set of paths that connect $i$ to $j$ and do not traverse any critical index,
with length bounded by $n$, and by $w(\Pi_2)$ the biggest weight of paths in $\Pi_2$. Denote by $\mu$ the second largest cycle geometric mean in $A$,
taken among the simple cycles only, then $\mu<1$.

Let us first prove that
\begin{equation}
\label{e:weakcsr}
a_{ij}^{(t)}\leq w(\Pi_1)\oplus w(\Pi_2)\mu^{t-n}
\end{equation}
for all $t$.

{\bf Case 1.} Suppose that $a_{ij}^{(t)}$ is the weight of a path $P$ that connects $i$ to $j$ via a critical index $k$.
Then it can be decomposed as $P_1\circ P_2$ where $P_1$ connects $i$ to $k$ and $P_2$ connects $k$
to $j$. Repeatedly applying the cycle deletion  to $P_1$ and $P_2$ we obtain a simple path $P'_1$
connecting $i$ to $k$ and a simple path $P'_2$ connecting $k$ to $j$. We conclude that $P'_1\circ P'_2\in\Pi_1$,
and
$$
a_{ij}^{(t)}=w(P)=w(P_1\circ P_2)\leq w(P'_1\circ P'_2)\leq w(\Pi_1).
$$
Note that since there exist paths connecting $i$ to $j$ through a critical index,
the cycle deletion argument above implies that $\Pi_1$ is non-empty and $w(\Pi_1)>0$.

{\bf Case 2.} Suppose that $a_{ij}^{(t)}$ is the weight of a path $\Tilde{P}$ that connects $i$ to $j$ not
traversing any critical index. Repeatedly applying the cycle deletion to $\Tilde{P}$ we obtain a simple path
or a cycle $\Tilde{P}_2$ with length bounded by $n$ and weight bounded by $w(\Pi_2)$. Since we deleted only non-critical cycles with cycle mean
not exceeding $\mu$, we obtain
$$
a_{ij}^{(t)}=w(\Tilde{P})\leq w(\Tilde{P}_2)\mu^{t-l(\Tilde{P}_2)}\leq w(\Pi_2)\mu^{t-n}.
$$

So in this case we also have
$a_{ij}^{(t)}\leq w(\Pi_1)\oplus w(\Pi_2)\mu^{t-n}$ for all $t$.

We need to show that $a_{ij}^{(t)}=w(\Pi_1)$ for all large enough $t$.
The path attaining $w(\Pi_1)$ (composed of two simple paths)
goes through a critical index $k$, which lies in a primitive critical component. Then by Lemma~\ref{l:schur} for all large enough
$t$ there exist critical cycles of length $t$ passing through $k$, hence for all large enough
$t$ there exist paths connecting $i$ to $j$ via $k$ with weight $w(\Pi_1)$. This shows that $a_{ij}^{(t)}=w(\Pi_1)$
for all large enough $t$.

As $i$ and $j$ were chosen only with the restriction that $j$ is in a spectral class with
eigenvalue $1$, and $i$ has access to $j$,
it follows that all columns of $A^t$ with indices in a spectral class with m.c.g.m. $1$ are ultimately constant (that is, ultimately periodic with period $1$ and growth rate $1$) , when the critical graph of $A$ is primitive and $\rho(A)=1$.

Consider the general case, that is, the case of an index $j$ in
a general spectral class with general m.c.g.m. $\rho$ and with general
cyclicity $\sigma=\sigma_{\rho}$ of $\crit(A_{\rho})$.  We first
use the reduction to the powers of $A_{\rho}$, see Theorem~\ref{t:reduction}(i), where $1$ is the
greatest eigenvalue. Raising $A$ to the power $\sigma$ and using  Theorem~\ref{t:samespectrum} (ii), which shows that the set of indices in spectral classes corresponding to a given eigenvalue
does not depend on power, we reduce to the main case considered above. In particular,
by Lemmas~\ref{l:sameperron} and~\ref{l:critpowers} the critical graph
$\crit((A_{\rho})^{\sigma})$ is primitive.

We obtain that the sequence of the $j$th columns of the powers of $A^{\sigma}$ is ultimately periodic with period $1$ and growth rate $\rho^{\sigma}$,
which implies that the sequence of the $j$th columns
of the powers of $A$ is ultimately periodic with growth rate $\rho$ and period
at most $\sigma$. Indeed, multiplying the equality
$A_{\cdot j}^{(l+1)\sigma}=A_{\cdot j}^{l\sigma}$ by $A^s$ for any $s\geq 1$ from the left,
we get $A_{\cdot j}^{(l+1)\sigma+s}=A_{\cdot j}^{l\sigma+s}$ for all $s\geq 1$. The proof is complete.
\end{proof}

We will need the following observation on support of the vectors belonging to the
core.

\begin{proposition}
\label{p:top-spectral}
Let $z\in\core(A)$, then
\begin{itemize}
\item[(i)] For each class $\mu$ of $A^{\sigma}$, either $N_{\mu}\subseteq\supp z$,
or $N_{\mu}\cap\supp z=\emptyset$.
\item[(ii)] Consider the set of classes $\mu$ of $A^{\sigma}$
such that $N_{\mu}\subseteq\supp z$, and let $R_z$ be the subgraph of the marked reduced graph $R(A^{\sigma})$ consisting only of such classes and edges between them. Then in
$R_z$, all final classes are spectral.
\end{itemize}
\end{proposition}
\begin{proof}
Part (i) holds for each eigenvector of $A^{\sigma}$, and in particular,
for every fundamental eigenvector $(A_{\rho})^*_{\cdot i}$ where
$\rho\in\Lambda(A)$ and $i$ belongs to the corresponding critical graph.
Part (ii) also holds for every fundamental eigenvector, since in $R_z$ there is a unique
final class $\mu$ with $N_{\mu}\subseteq\supp z$ containing $i$, and it is spectral.

To deduce both claims for general $z\in\core(A)$, we notice that by
Theorem~\ref{t:core}, $\core(A)$ is generated by such fundamental eigenvectors, and that taking max-linear combinations corresponds to taking the union of supports.
\end{proof}

\begin{proposition}
\label{p:colsclasses}
\begin{itemize}
\item[(i)] If $i$ belongs to a spectral class of $A$ then $A^t_{\cdot i}\in\coremax(A)$ for
all large enough $t$;
\item[(ii)] If $i$ belongs to a non-trivial non-spectral class of $A$ then
$A^t_{\cdot i}\notin\core(A)$ for any $t$.
\end{itemize}
\end{proposition}
\begin{proof}
(i): Follows by Theorem~\ref{t:cyclicity}.\\
(ii): We use Proposition~\ref{p:top-spectral}. For any $t$,
either the alternative ``$N_{\mu}\subseteq \supp(A^t_{\cdot i})$ or
$N_{\mu}\cap \supp(A^t_{\cdot i})=\emptyset$'' does not hold for some
$\mu$, or it holds but
if for $z=A^t_{\cdot i}$ we take the subgraph $R_z$ described in Proposition~\ref{p:top-spectral} part (ii), then the only final class of this subgraph
is the one containing $i$, and it is not spectral.
Thus we have $A^t_{\cdot i}\notin\core(A)$ for any $t$.
\end{proof}

\begin{theorem}
\label{t:core-finite}
Finite stabilization of $\coremax(A)$ occurs if and only if all nontrivial classes of $A$ are spectral.
\end{theorem}
\begin{proof}
The ``only if'' part: If there is a nontrivial non-spectral class then the columns
$A^t_{\cdot i}$ with $i$ in that class never belong to $\coremax(A)$ by Proposition~\ref{p:colsclasses}.

The ``if'' part: Assume there are only spectral classes and trivial
classes. By Theorem~\ref{t:cyclicity}, there is an integer $T(A)$
such that all sequences $\{A^t_{\cdot i}\}_{t\geq T(A)}$, where $i$
is an index in a spectral class, are periodic. Now take $t\geq
T(A)+n$ and let $i$ be an index of a trivial class. Let $\{P_k,\
k=1,\ldots, r\}$ be all paths connecting certain indices $j_k$ in
spectral classes to $i$ with a property that in these paths $j_k$ is
not repeated and no other index than $j_k$ can belong to a
non-trivial class. Such paths will be called {\em direct access
paths}. Their length is bounded by $n-1$, for otherwise they contain
a cycle, and any cycle belongs to a non-trivial class. So indeed,
there is only a finite number of such paths. Let $l(P_k)$, $w(P_k)$
denote the length, resp. the weight of $P_k$, then for all $s$ for
which there are no direct access paths with length $t$ connecting
$s$ to $i$ we can express
\begin{equation}
\label{e:threads}
a_{si}^{(t)}=\bigoplus_{k=1}^r a_{sj_k}^{(t-l(P_k))} w(P_k).
\end{equation}
Theorem~\ref{t:cyclicity} implies that $\{a_{sj_k}^{(t-l(P_k))}\}_{t>l(P_k)}$ is periodic at $t\geq T(A)+l(P_k)$,
so every such sequence is periodic at $t\geq T(A)+n$.
There are no direct access paths of length $t\geq n$, and hence of length
$t\geq T(A)+n$. Since for $t\geq T(A)+n$ all columns $A_{\cdot j_k}^{t-l(P_k)}$
are in $\core(A)$, using~\eqref{e:threads} we see that also $A^t_{\cdot i}\in\coremax(A)$ for $t\geq T(A)+n$.
\end{proof}

\subsection{Robustness}
\label{ss:robust}

We characterize the classes of robust and orbit periodic matrices
($\goodclass_3$ and $\goodclass_4$ of Subsect.~\ref{ss:maxalg-easy}) by means of the
restriction of these properties to the core.

Matrix $A$ is called {\bf core robust} if for any $x\in\coremax(A)\backslash\{0\}$ (instead of any
$x\in\Rpn\backslash\{0\}$) the orbit
$\{A^tx,\ t\geq 1\}$ hits an eigenvector of $A$ in finite number of steps.

\begin{theorem}
\label{t:robust-core}
$A\in\Rpnn$ with no zero columns is robust if and only if it is core robust and the core finitely stabilizes.
\end{theorem}
\begin{proof}
Evidently, core robustness is necessary.
For the finite stabilization observe that the sequence of
columns $\{A^t_{\cdot i},\; t\geq 0\}$ is the orbit of the $i$th unit vector, and hence if $A$
is robust then this orbit converges to an eigenvector of $A$, so
$A_{\cdot i}^t\in\core(A)$ for all $i$ and sufficiently large $t$.

The properties are also sufficient. Indeed, any orbit $\{A^t x,\;t\geq 0\}$, for $x\neq 0$, first gets to
a vector of the core by the finite stabilization
property, and this vector is nonzero since $A$ does not have zero columns. Then the orbit
hits an eigenvector of $A$ in a finite number of steps, by the core robustness.
\end{proof}

\begin{theorem}
\label{t:core-robust}
$A$ is core robust if and only if
\begin{itemize}
\item[(i)] the mapping induced by $A$ on the scaled extremals of its core
is identity; i.e., $\core(A)$ is generated by the eigenvectors of $A$.
\item[(ii)] $x\in \vmax(A,\rho_1)\backslash\{0\}$, $y\in \vmax(A,\rho_2)\backslash\{0\}$ and $\rho_1<\rho_2$ imply that $\supp(x)\subseteq\supp(y)$.
\end{itemize}
\end{theorem}
\begin{proof}
Both (i) and (ii) are necessary. Indeed, the mapping induced by $A$ on the
scaled extremals of the core is a bijection, so the orbits of all extreme rays are periodic.
If (i) does not hold then there are extreme rays with period
greater than $1$, and they never stabilize.  If (ii) does not hold, then no orbit of any combinations $\alpha x\oplus\beta y$ with
$\alpha>0$ and $\beta>0$ and $x,y$ as in (ii) ever converges to an eigenvector.
Indeed, we have
\begin{equation}
A^t(\alpha x\oplus\beta y)=\alpha A^t x\oplus\beta A^t y =
\alpha \rho_1^t x\oplus \beta\rho_2^t y.
\end{equation}
For $i\in\supp x\backslash\supp y$ we have
$A^t(\alpha x\oplus\beta y)_i=\alpha\rho_1^t x_i$ and for
$i\in\supp y$ we have $A^t(\alpha x\oplus \beta y)_i=
\alpha\rho_1^t x_i\oplus \beta\rho_2^t y_i=\beta\rho_2^t y_i$ for all large
enough $t$. Hence an eigenvector will never be reached, a contradiction.

Let (i) and (ii) hold, then any $y\in\coremax(A)$ can be written as $y=\bigoplus_{\mu}\alpha_{\mu} x_{\mu},$ where
$x_{\mu}$ are eigenvectors corresponding to the eigenvalues
$$
\rho_1\leq\rho_2\leq\ldots\leq\rho_l=\ldots=\rho_m.
$$
It follows that the union of all supports corresponding to the vectors with the largest eigenvalue
contains all other supports. Hence
$$
A^ty=\bigoplus_{\mu=1}^m \alpha_{\mu}\rho_{\mu}^t x_{\mu}=
\rho_m^t\bigoplus_{\mu=l}^m \alpha_{\mu} x_{\mu}
$$
at all large enough $t$, hence $A$ is core robust.
\end{proof}

Now we show how to deduce the characterization of robust matrices obtained
by Butkovi\v{c}, Cuninghame-Green, Gaubert~\cite{BCG-09}.

\begin{theorem}[Butkovi{\v{c}} et al.~\cite{BCG-09}]
$A\in\Rpnn$ with no zero columns is robust if and only if the following
conditions hold:
\begin{itemize}
\item[(i)] All nontrivial classes of Frobenius normal form are spectral;
\item[(ii)] The critical graphs associated with all $\rho\in\Lambda(A)$
are primitive;
\item[(iii)] For any two classes $\mu$ and $\nu$, if both
$\mu\not\to\nu$ and $\nu\not\to\mu$ then $\rho_{\mu}=\rho_{\nu}$.
\end{itemize}
\end{theorem}
\begin{proof}
By Theorem~\ref{t:core-finite}, part~(i) is equivalent to finite stabilization of
$\core(A)$. In view of Theorems~\ref{t:robust-core} and Theorem~\ref{t:core-robust}, it is sufficient to show that part~(ii) is equivalent with
Theorem~\ref{t:core-robust}(i), and part~(iii) is equivalent with
Theorem~\ref{t:core-robust}(ii).

For the first equivalence, we observe that
 $V(A^{\sigma_{\Lambda}})=V(A)$ if and only if
$\sigma_{\Lambda}=1$. Indeed, the ``if'' part is trivial, and for the ``only if'' part
recall that $V(A)\subseteq V(A^t)\subseteq V(A^{\sigma_{\Lambda}})$ for all $t\geq 1$ so that
if $V(A^{\sigma_{\Lambda}})=V(A)$ then the period of the sequence $\{V(A^t)\}_{t\geq 1}$
is $1$, which equals $\sigma_{\Lambda}$ by Theorem~\ref{t:girls-general}.

For the second
equivalence, observe using the description of spectral classes of
Theorem~\ref{t:spectrum}, that part~(iii) of the present theorem is equivalent to the following
condition:\\
(iii'): For any two classes $\mu$ and $\nu$, $\rho_{\mu}<\rho_{\nu}$ implies $\mu\to\nu$.\\
It suffices to show that this condition is equivalent to the one of
Theorem~\ref{t:core-robust}(ii).

For this, applying Theorem~\ref{t:FVmaxalg}(i'), observe that (iii') holds if and only if
for any pair of (spectral) classes $\mu$ and $\nu$ with $\rho_{\mu}<\rho_{\nu}$ and for any
pair of critical components $\Tilde{\mu}$ in $\mu$ and $\Tilde{\nu}$ in $\nu$, the associated fundamental
eigenvectors satisfy $\supp(x_{\Tilde{\mu}})\subseteq \supp(y_{\Tilde{\nu}})$.

This already shows that if (iii') is violated then Theorem~\ref{t:core-robust}(ii) does
not hold. For the converse, we argue that
if $x\in V(A,\rho_1)$ and $y\in V(A,\rho_2)$, where $x,y\neq 0$ and $\rho_1<\rho_2$, then $x$, resp. $y$
are combinations of fundamental eigenvectors $x_{\Tilde{\mu}}$, resp. $y_{\Tilde{\nu}}$
satisfying $\supp(x_{\Tilde{\mu}})\subseteq \supp(y_{\Tilde{\nu}})$. As $\supp(x)$, resp.
$\supp(y)$ is the union of supports of such $x_{\Tilde{\mu}}$, resp. $y_{\Tilde{\nu}}$,
the inclusion $\supp(x)\subseteq\supp(y)$ follows.
\end{proof}

Recall the notion of orbit periodicity defined in Subsection~\ref{ss:maxalg-easy}:
$A$ is called orbit periodic if the orbit of any vector $x\in\Rpn\backslash\{0\}$
is ultimately periodic, which happens if and only if it hits an eigenvector of $A^{\sigma_{\Lambda}}$.
The core restriction of orbit periodicity is defined as follows: $A$
is called {\em core periodic} if the orbit of any vector $x\in\coremax(A)\backslash\{0\}$
is ultimately periodic, i.e., it hits an eigenvector
of $A^{\sigma_{\Lambda}}$.

Observe that $A$ is orbit periodic
if and only if $A^{\sigma}$ is robust, and this also holds for the core restrictions (more generally, for
any restrictions) of both notions. The next observations are analogous to
Theorems~\ref{t:robust-core} and~\ref{t:core-robust}.

\begin{theorem}
\label{t:periodic-core}
$A\in\Rpnn$ with no zero columns is orbit periodic if and only if it is core periodic and the core finitely stabilizes.
\end{theorem}
\begin{proof}
Follows the lines of the proof of Theorem~\ref{t:robust-core}.
\end{proof}

\begin{theorem}
\label{t:core-periodic}
$A\in\Rpnn$ is core periodic if and only if
$x\in \vmax(A^{\sigma_{\Lambda}},\rho_1^{\sigma_{\Lambda}})\backslash\{0\}$, $y\in \vmax(A^{\sigma_{\Lambda}},\rho_2^{\sigma_{\Lambda}})\backslash\{0\}$ with $\rho_1,\rho_2\in\Lambda(A)$ and $\rho_1<\rho_2$ imply that
$\supp(x)\subseteq\supp(y)$.
\end{theorem}
\begin{proof}
$A\in\Rpnn$ is core periodic if and only if $A^{\sigma}$ is core robust. Observe that $A^{\sigma}$ satisfies
condition (i) in Theorem~\ref{t:core-robust} since $\core(A)=\core(A^{\sigma})$,
and this is generated by the eigenvectors of $A^{\sigma}$. Hence core robustness of $A^{\sigma}$ is equivalent to
condition (ii) in Theorem~\ref{t:core-robust}. As
$\Lambda(A^{\sigma})=\{\rho^{\sigma}\mid \rho\in\Lambda(A)\}$, this condition is identical to the one of the present proposition.
\end{proof}

\subsection{Bijection and weak stability}

The concept of a {\em weakly stable matrix} was introduced by Butkovi\v{c} et al.~\cite{BSS-ws}.
It requires that the orbit of any vector $x$ {\bf does not} hit any eigenvector of $A$, unless
that vector $x$ already is an eigenvector. Equivalently we can write
\begin{equation}
\label{e:weaklystable}
Ay=z,\ z\in \vmax(A,\rho)\Rightarrow y=\rho^{-1}z.
\end{equation}
Weakly stable matrices can be characterized as follows~\cite{BSS-ws}:
each spectral class is initial and each critical graph is a Hamiltonian cycle.

We now study the {\em core weakly stable matrices}:
restricting vector $x$ in the above definition to $\core(A)$. Thus $A$ is core weakly stable if and only if
\begin{equation}
\label{e:weaklystable-core}
Ay=z,\ z\in \vmax(A,\rho)\ \text{and}\ y\in\coremax(A)\Rightarrow y=\rho^{-1}z.
\end{equation}

As we will shortly see, the core weak stability is equivalent to the matrix being bijective
on its core. We will need the following corollary of Lemma~\ref{l:sameaccess} and
Theorem~\ref{t:samespectrum}

\begin{corollary}
\label{c:sameaccess}
Let $\mu$ and $\nu$ be classes of $A$ with cyclicities $\sigma_{\mu}$
and $\sigma_{\nu}$, and let $\mu^{(k)}$ for $k=1,\ldots,$ gcd$(t,\sigma_{\mu})$
and $\nu^{(l)}$ for
$l=1,\ldots,$ gcd$(t,\sigma_{\nu})$ be the classes of $A^t$ derived from
$\mu$ and $\nu$ respectively. Then the following are equivalent:
\begin{itemize}
\item[(i)]  $\mu$ and $\nu$ are spectral and $\mu\to\nu$ (in $A$);
\item[(ii)]  all derived classes
$\mu^{(k)}$ and $\nu^{(l)}$ are spectral, and for each $k$ there is an $l$
and for each $l$ there is a $k$ such that $\mu^{(k)}\to\nu^{(l)}$ (in $A^t$);
\item[(iii)]  there exists a pair of derived classes
$\mu^{(k)}$ and $\nu^{(l)}$ that are spectral, such that $\mu^{(k)}\to\nu^{(l)}$
(in $A^t$).
\end{itemize}
\end{corollary}

Let us also recall the following well-known observation, which
we briefly prove for the reader's convenience.

\begin{lemma}[\cite{CGG-99} Lemma~1.4 part~4]
\label{l:zeq}
Let $x\in V(A,\rho)$, and let $x_{C}$ be the subvector of $x$ extracted from
the node set of the critical graph associated with $\rho$. Then $x$ is uniquely determined by $x_C$.
\end{lemma}
\begin{proof}
Using Proposition~\ref{p:vamrho}, assume without loss of generality that $\rho=1$ is the greatest eigenvalue of $A$ (in other words, the
maximum cycle geometric mean). The set of the critical nodes of $A$ induces the following block decomposition:
\begin{equation}
\label{e:ablock}
A=
\begin{pmatrix}
A_{CC} & A_{CN}\\
A_{NC} & A_{NN}
\end{pmatrix}
\end{equation}
with $A_{CC}$, resp. $A_{NN}$ being the principal submatrices extracted from the set of critical nodes, resp. its complement
(the noncritical nodes). For the subvector $x_N$ extracted from the noncritical nodes, we have
\begin{equation}
\label{e:bellman}
x_N=A_{NC} x_C\oplus A_{NN} x_N.
\end{equation}
For a given $x_C$, this is a max-algebraic Z-matrix (or Bellman) equation on $x_N$ as treated, for instance, in~\cite{BSS-zeq}.
As $A_{NN}$ has $\rho(A_{NN})<1$,  we obtain that $(A_{NN})^*A_{NC} x_C$ is the only solution,
hence $x_N$ is uniquely determined by $x_C$.
\end{proof}

\begin{theorem}
\label{t:bijection}
Let $A\in\Rpnn$. The following are equivalent:
\begin{itemize}
\item[(i)] the mapping induced by $A$ on its core is a bijection;
\item[(ii)] $A$ is core weakly stable;
\item[(iii)] if $\mu$ and $\nu$ are spectral classes of $A$ and $\mu\to\nu$ then $\rhomax_{\mu}=\rhomax_{\nu}$. In other words, spectral classes with different Perron roots do not access each other.
\end{itemize}
Moreover, if $A$ satisfies (i),(ii), or (iii) then $A^t$ satisfies
(i),(ii), or (iii), respectively, for every $t\geq 1$.
Also if $A^t$ satisfies (i),(ii), or (iii) for some $t\geq 1$ then the same
properties hold
for $A$.
\end{theorem}
\begin{proof}
We show first that each of claims (i) and (iii) holds for $A$ if and only if it
holds for $A^t$ with an arbitrary $t$.
Then, the implications (i)$\Rightarrow$(ii) and (ii)$\Rightarrow$(iii)
will be shown for arbitrary $t\geq 1$, and the implication (iii)$\Rightarrow$(i)
will be shown for $t=\sigma_{\Lambda}$.

First observe that $\core(A^t)=\core(A)$ for all $t$. Then for (i),
the equivalence for all $t$ is clear.

To show the same equivalence for (iii) we use
Corollary~\ref{c:sameaccess} and argue by contradiction. By that Corollary, if $\mu$ and $\nu$ are
spectral in $A$ with $\mu\to\nu$ but $\rhomax_{\mu}\neq\rhomax_{\nu}$, then
for each $A^t$ there exist derived spectral classes $\mu'$ and $\nu'$ in $A^t$
such that $\mu'\to\nu'$ in $A^t$ and $\rhomax_{\mu'}\neq\rhomax_{\nu'}$. So if (iii) does not hold for $A$ then it does not hold for $A^t$. The
converse implication follows similarly.

Also note that by Lemma~\ref{l:actioncore} (ii) we know that $A$ and hence all its powers are surjective on the core.

(i)$\Rightarrow$ (ii) for every $A^t$: Assume by contradiction that the core weak stability~\eqref{e:weaklystable}, formulated for $A^t$ as
\begin{equation}
\label{e:weaklystables}
A^ty=z,\ z\in \vmax(A^t,\rho^t)\ \text{and}\ y\in\coremax(A^t)\Rightarrow y=\rho^{-t}z,
\end{equation}
is violated. Since $y=\rho^{-t}z$ satisfies $A^t y=z$, this means that
there is $z\in\coremax(A^t)=\coremax(A)$ and $y'\in\coremax(A)$ other than
$\rho^{-t}z$ such that $A^ty'=z$. This means that $A^t$ and hence $A$ are not bijective on the core.

(ii)$\Rightarrow$ (iii) for all $A^t$, assuming w.l.o.g. $t=1$:
If (iii) does not hold, then there exist $\mu,\nu$ spectral with $\mu\to\nu$ such that
$\rhomax_{\mu}<\rhomax_{\nu}$. Then by Theorem~\ref{t:FVmaxalg}(i') $\core(A)$
contains two eigenvectors
$x_{\mu}$ and $x_{\nu}$ with $\supp(x_{\mu})\subset\supp(x_{\nu})$
corresponding to the Perron roots $\rhomax_{\mu}<\rhomax_{\nu}$.
Taking $\beta\neq 0$ and $\alpha$ sufficiently large, we make a combination $\alpha x_{\mu}\oplus\beta x_{\nu}$,
different from both $\alpha x_{\mu}$ and $\beta x_{\nu}$. However,
$$ A^t (\alpha x_{\mu}\oplus\beta x_{\nu})=\alpha\rhomax_{\mu}^t x_{\mu}\oplus
\beta\rhomax_{\nu}^t x_{\nu}=\beta\rhomax_{\nu}^t x_{\nu},$$
starting from some $t$, but not for all $t$. So $x_{\nu}$ has a preimage in the core which is not an eigenvector, and $A$ is not core weakly stable.

(iii)$\Rightarrow$ (i) for $s=\sigma_{\Lambda}$. Let $\sigma:=\sigma_{\Lambda}$,
for brevity.
We need to show that if that if $A^{\sigma}z=A^{\sigma}z'$ for $z,z'\in\core(A)$
then $z=z'$. We can write $z=\bigoplus_{\mu} x_{\mu}$ and
$z'=\bigoplus_{\mu} y_{\mu}$ where $x_{\mu}$, $y_{\mu}$ are eigenvectors
of $A^{\sigma}$ corresponding to $\rho_{\mu}^{\sigma}$ or zero vectors, with
$\mu$ running over all eigenvalues of $A$. Then we have
\begin{equation}
A^{\sigma}\left(\bigoplus_{\mu} x_{\mu}\right)=\bigoplus_{\mu} \rho_{\mu}^{\sigma} x_{\mu},\
A^{\sigma}\left(\bigoplus_{\mu} y_{\mu}\right)=\bigoplus_{\mu} \rho_{\mu}^{\sigma} y_{\mu},
\end{equation}
and our goal is to show that
\begin{equation}
\label{e:xmu=ymu}
\bigoplus_{\mu} \rhomax_{\mu}^{\sigma} x_{\mu}= \bigoplus_{\mu} \rhomax_{\mu}^{\sigma} y_{\mu}\Rightarrow\\
x_{\mu}=y_{\mu}\quad\forall \mu,
\end{equation}
which implies $z=z'$. For~\eqref{e:xmu=ymu}, observe that by Theorem~\ref{t:FVmaxalg}(i') and the assumed condition (iii),
the support of any vector in $V(A^{\sigma},\rho^{\sigma})$ contains only indices in
spectral classes with $\rho^{\sigma}$ or indices in non-spectral and trivial classes. Then the {\em spectral part of support}, i.e., the part consisting of
all indices in spectral classes, is disjoint from the spectral part of support
of any vector in $V(A^{\sigma},\Tilde{\rho}^{\sigma})$ when $\Tilde{\rho}\neq\rho$.
Therefore in any max-algebraic sum of $x_{\mu}\in \vmax(A^{\sigma},\rhomax_{\mu}^{\sigma})$
with different $\rhomax_{\mu}$, as in~\eqref{e:xmu=ymu}, the subvectors
of $x_{\mu}$ in spectral classes are determined uniquely. This shows that the
components of $x_{\mu}$ and $y_{\mu}$ in spectral classes associated with
$\rhomax_{\mu}$ are the same.

By Lemma~\ref{l:zeq} any eigenvector
is uniquely determined by its critical components, and hence $x_{\mu}$ and
$y_{\mu}$ in~\eqref{e:xmu=ymu} are uniquely determined by their components in
spectral classes, implying that $x_{\mu}=y_{\mu}$ for all $\mu$, which
is~\eqref{e:xmu=ymu}. Then $z=z'$, and the proof is complete.
\end{proof}

%\bibliographystyle{plain}
%\bibliography{mariecurie}

\end{document}